\documentclass[12pt]{amsart}
\usepackage{eurosym}
\usepackage{graphicx,amscd,color,amsmath,amsfonts,amssymb,geometry}

\newtheorem{theorem}{Theorem}[section]
\newtheorem{proposition}[theorem]{Proposition}

\newtheorem{remark}[theorem]{Remark}

\newtheorem*{theo}{Theorem}
\newtheorem*{gkszi}{Generalized Kahane--Salem--Zygmund inequality}
\numberwithin{equation}{section}

\begin{document}
\title[Optimal Hardy--Littlewood type inequalities ]{Optimal
Hardy--Littlewood type inequalities for $m$-linear forms on $\ell _{p}$
spaces with $1\leq p\leq m$}
\author[G. Ara\'{u}jo]{Gustavo Ara\'{u}jo}
\address{Departamento de An\'{a}lisis Matem\'{a}tico \\
\indent Facultad de Ciencias Matem\'{a}ticas \\
\indent Universidad Complutense de Madrid \\
\indent Plaza de Ciencias 3 \\
\indent Madrid, 28040, Spain.}
\email{gdasaraujo@gmail.com}
\author[D. Pellegrino]{Daniel Pellegrino}
\address{Departamento de Matem\'{a}tica \\
\indent Universidade Federal da Para\'{\i}ba \\
\indent 58.051-900 - Jo\~{a}o Pessoa, Brazil.}
\email{pellegrino@pq.cnpq.br and dmpellegrino@gmail.com}
\thanks{G. Ara\'{u}jo is supported by PDSE/CAPES 8015/14-7. D. Pellegrino is
supported by CNPq Grant 401735/2013-3 - PVE - Linha 2 and INCT-Matem\'{a}%
tica.}
\maketitle

\begin{abstract}
The Hardy--Littlewood inequalities for $m$-linear forms on $\ell _{p}$
spaces are stated for $p>m$. In this paper, among other results, we
investigate similar results for $1\leq p\leq m.$ Let $\mathbb{K}$ be $%
\mathbb{R}$ or $\mathbb{C}$ and $m\geq 2$ be a positive integer. Our main
results are the following sharp inequalities:

\begin{itemize}
\item[(i)] If $\left( r,p\right) \in \left( \lbrack 1,2]\times \lbrack
2,2m)\right) \cup \left( \lbrack 1,\infty )\times \lbrack 2m,\infty \right]
) $, then there is a positive constant $D_{m,r,p}^{\mathbb{K}}$ (not
depending on $n$) such that 
\begin{equation*}
\textstyle\left( \sum\limits_{j_{1},...,j_{m}=1}^{n}\left\vert
T(e_{j_{1}},...,e_{j_{m}})\right\vert ^{r}\right) ^{\frac{1}{r}}\leq
D_{m,r,p}^{\mathbb{K}}n^{\max \left\{ \frac{2mr+2mp-mpr-pr}{2pr},0\right\}
}\left\Vert T\right\Vert
\end{equation*}%
for all $m$--linear forms $T:\ell _{p}^{n}\times \cdots \times \ell
_{p}^{n}\rightarrow \mathbb{K}$ and all positive integers $n$.

\item[(ii)] If $\left( r,p\right) \in \lbrack 2,\infty )\times (m,2m]$, then 
\begin{equation*}
\textstyle\left( \sum\limits_{j_{1},...,j_{m}=1}^{n}\left\vert
T(e_{j_{1}},...,e_{j_{m}})\right\vert ^{r}\right) ^{\frac{1}{r}}\leq \left( 
\sqrt{2}\right) ^{m-1}n^{\max \left\{ \frac{p+mr-rp}{pr},0\right\}
}\left\Vert T\right\Vert
\end{equation*}%
for all $m$--linear forms $T:\ell _{p}^{n}\times \cdots \times \ell
_{p}^{n}\rightarrow \mathbb{K}$ and all positive integers $n.$
\end{itemize}

Moreover the exponents $\max \{(2mr+2mp-mpr-pr)/2pr,0\}$ in (i) and $\max
\{(p+mr-rp)/pr,0\}$ in (ii) are optimal. The cases $\left( r,p\right)
=\left( 2m/\left( m+1\right) ,\infty \right) $ and $\left( r,p\right)
=\left( 2mp/\left( mp+p-2m\right) ,p\right) $ for $p\geq 2m$ and $\left(
r,p\right) =\left( p/\left( p-m\right) ,p\right) $ for $m<p<2m$ recover the
classical Bohnenblust--Hille and Hardy--Littlewood inequalities.
\end{abstract}

\section{Introduction}

The recent years witnessed an intense interest in the Bohnenblust--Hille
inequality and its applications in Complex Analysis, Analytic Number Theory
and Quantum Information Theory. The Bohnenblust--Hille inequality was proved
in 1931, in the \textit{Annals of Mathematics}, as a crucial tool to prove
the Bohr's absolute convergence problem on Dirichlet series. Surprisingly,
this inequality was overlooked for almost 80 years and rediscovered some
years ago. Since then, it has been used in different areas of Mathematics
and several challenging problems remain open. From now on $\mathbb{K}$
denotes the real scalar field $\mathbb{R}$ or the complex scalar field $%
\mathbb{C}$.

\begin{theo}[Bohnenblust and Hille \protect\cite{bh}, 1931]
For any positive integer $m\geq 2$, there exists a constant $B_{\mathbb{K}%
,m}\geq 1$ such that 
\begin{equation}
\textstyle\left( \sum\limits_{j_{1},...,j_{m}=1}^{n}\left\vert
T(e_{j_{1}},...,e_{j_{m}})\right\vert ^{\frac{2m}{m+1}}\right) ^{\frac{m+1}{%
2m}}\leq B_{\mathbb{K},m}\left\Vert T\right\Vert  \label{tttt}
\end{equation}%
for all $m$--linear forms $T:\ell _{\infty }^{n}\times \cdots \times \ell
_{\infty }^{n}\rightarrow \mathbb{K}$, and all positive integers $n$.
Moreover, the exponent $2m/(m+1)$ is optimal.
\end{theo}

For references we mention, for instance, \cite{alb, bohr, deff2, deff, Nuuu}
and the very interesting survey \cite{sur}. The optimal values of $B_{%
\mathbb{K},m}$ are unknown; the best known upper and lower estimates for the
constants in (\ref{tttt}) are (see \cite{bohr} and \cite{diniz}): 
\begin{equation*}
\begin{array}{rcccll}
&  & B_{\mathbb{C},m} & \leq & \textstyle\prod\limits_{j=2}^{m}\Gamma \left(
2-\frac{1}{j}\right) ^{\frac{j}{2-2j}}<m^{\frac{1-\gamma }{2}}<m^{0.22}, & 
\vspace{0.2cm} \\ 
2^{1-\frac{1}{m}} & \leq & B_{\mathbb{R},m} & \leq & 2^{\frac{446381}{55440}%
- \frac{m}{2}}\textstyle\prod\limits_{j=14}^{m}\left( \frac{\Gamma \left( 
\frac{3}{2}-\frac{1}{j}\right) }{\sqrt{\pi }}\right) ^{\frac{j}{2-2j}%
}<1.3\cdot m^{\frac{2-\log 2-\gamma }{2}} & \text{if }m\geq 14,\vspace{0.2cm}
\\ 
2^{1-\frac{1}{m}} & \leq & B_{\mathbb{R},m} & \leq & \textstyle%
\prod\limits_{j=2}^{m}2^{\frac{1}{2j-2}}<1.3\cdot m^{\frac{2-\log 2-\gamma }{%
2}}<1.3\cdot m^{0.37} & \text{if }m\leq 13,%
\end{array}%
\end{equation*}
where $\gamma $ denotes the Euler--Mascheroni constant.

\bigskip A natural question is: what happens if we replace $\ell _{\infty
}^{n}$ by $\ell _{p}^{n}$ in the Bohnenblust--Hille inequality? This
question was answered by Hardy and Littlewood (see \cite{hardy}) in 1934 for
bilinear forms, and complemented by Praciano-Pereira (see \cite{pra}) in
1981 for $m$-linear forms and $p\geq 2m$ (and later by Dimant and
Sevilla-Peris for $m<p<2m$ (see \cite{dimant})).

\begin{theo}[Hardy--Littlewood/Praciano-Pereira \protect\cite{hardy,pra},
1934/1981]
Let $m\geq 2$ be a positive integer and $p\geq 2m.$ For all $m$--linear
forms $T:\ell _{p}^{n}\times \cdots \times \ell_{p}^{n}\rightarrow \mathbb{K}
$ and all positive integers $n$, 
\begin{equation}
\textstyle\left( \sum\limits_{j_{1},...,j_{m}=1}^{n}\left\vert
T(e_{j_{1}},...,e_{j_{m}})\right\vert ^{\frac{2mp}{mp+p-2m}}\right) ^{\frac{%
mp+p-2m}{2mp}}\leq \left( \sqrt{2}\right) ^{m-1}\left\Vert T\right\Vert.
\label{yd}
\end{equation}
Moreover, the exponent $2mp/(mp+p-2m)$ is optimal.
\end{theo}

\begin{theo}[Hardy--Littlewood/Dimant--Sevilla-Peris \protect\cite{hardy,
dimant}, 1934/2014]
Let $m\geq 2$ be a positive integer and $m<p<2m$. For all $m$--linear forms $%
T:\ell_{p}^{n}\times \cdots \times \ell _{p}^{n}\rightarrow \mathbb{K}$ and
all positive integers $n$, 
\begin{equation*}
\textstyle\left( \sum\limits_{j_{1},...,j_{m}=1}^{n}\left\vert
T(e_{j_{1}},...,e_{j_{m}})\right\vert ^{\frac{p}{p-m}}\right) ^{\frac{p-m}{p}%
}\leq \left( \sqrt{2}\right) ^{m-1}\left\Vert T\right\Vert.
\end{equation*}
Moreover, the exponent $p/(p-m)$ is optimal.
\end{theo}

\bigskip From now on, if $f$ is a function, we define $f(\infty
):=\lim_{p\rightarrow \infty }f(p)$ whenever it makes sense. In this fashion
note that the Hardy--Littlewood/Praciano-Pereira inequality encompasses the
Bohnenblust--Hille inequality.

To the best of our knowledge, the case $p\leq m$ was only explored for the
case of Hilbert spaces ($p=2,$ see \cite[Corollary 5.20]{compl} and \cite%
{cobos}) and the case $p=\infty$ was explored in \cite{camposar}. In \cite[Corollary 5.20]{compl} it is shown that for $p=2$ the
inequality has an extra power of $n$ in its right hand side. Other natural
questions are how the the Hardy--Littlewood/Praciano-Pereira and
Hardy--Littlewood/Dimant--Sevilla-Peris theorems behave if we replace the
optimal exponents $2mp/(mp+p-2m)$ and $p/(p-m)$ by a smaller value $r$. More
precisely, what power of $n$ will appear, depending on $r,m,p$? Our main
results answer this question (see Theorem \ref{888}) and extends \cite[%
Corollary 5.20]{compl} to $1\leq p\leq m$ (see Theorem \ref{888}(a) and
Proposition \ref{prop1}).

The main result of this note is the following:


\begin{theorem}
\label{888} Let $m\geq 2$ be a positive integer.

\begin{itemize}
\item[(a)] If $\left( r,p\right) \in \left( \lbrack 1,2]\times \lbrack
2,2m)\right) \cup \left( \lbrack 1,\infty )\times \lbrack 2m,\infty ]\right) 
$, then there is a constant $D_{m,r,p}^{\mathbb{K}}>0$ (not depending on $n$%
) such that 
\begin{equation*}
\textstyle\left( \sum\limits_{j_{1},...,j_{m}=1}^{n}\left\vert
T(e_{j_{1}},...,e_{j_{m}})\right\vert ^{r}\right) ^{\frac{1}{r}}\leq
D_{m,r,p}^{\mathbb{K}}n^{\max \left\{ \frac{2mr+2mp-mpr-pr}{2pr},0\right\}
}\left\Vert T\right\Vert
\end{equation*}%
for all $m$--linear forms $T:\ell _{p}^{n}\times \cdots \times \ell
_{p}^{n}\rightarrow \mathbb{K}$ and all positive integers $n$. Moreover, the
exponent $\max \left\{ (2mr+2mp-mpr-pr)/2pr,0\right\} $ is optimal.

\item[(b)] If $\left( r,p\right) \in \lbrack 2,\infty )\times (m,2m],$ then
there is a constant $D_{m,r,p}^{\mathbb{K}}>0$ (not depending on $n$) such
that 
\begin{equation*}
\textstyle\left( \sum\limits_{j_{1},...,j_{m}=1}^{n}\left\vert
T(e_{j_{1}},...,e_{j_{m}})\right\vert ^{r}\right) ^{\frac{1}{r}}\leq
D_{m,r,p}^{\mathbb{K}}n^{\max \left\{ \frac{p+mr-rp}{pr},0\right\}
}\left\Vert T\right\Vert
\end{equation*}%
for all $m$--linear forms $T:\ell _{p}^{n}\times \cdots \times \ell
_{p}^{n}\rightarrow \mathbb{K}$ and all positive integers $n.$ Moreover, the
exponent $\max \left\{ (p+mr-rp)/pr,0\right\} $ is optimal.
\end{itemize}
\end{theorem}

\begin{remark}
The first item of the above theorem recovers \cite[Corollary 5.20(i)]{compl}
(just make $p=2$) and \cite[Proposition 5.1]{camposar}. 
\end{remark}

\section{The proof}

\subsection{First part: preparatory results}

We begin by recalling a generalization of the Kahane--Salem--Zygmund
inequality which is an extension of a result due to Boas (\cite{boas}) that
will be useful in the proof of the optimality of the exponents:

\begin{gkszi}[see \protect\cite{alb}]
Let $m,n\geq 1$, let $p\in \lbrack 1,\infty ]$, and let 
\begin{equation*}
\alpha (p)=\left\{ 
\begin{array}{ll}
\textstyle\frac{1}{2}-\frac{1}{p} & \text{ if }p\geq 2\vspace{0.2cm} \\ 
0 & \text{ otherwise.}%
\end{array}%
\right.
\end{equation*}%
There is an universal constant $C_{m}$ (depending only on $m$) and there
exists an $m$-linear form $A:\ell _{p}^{n}\times \dots \ell
_{p}^{n}\rightarrow \mathbb{K}$ of the form 
\begin{equation}
\textstyle A(z^{(1)},\dots ,z^{(m)})=\sum\limits_{i_{1},\dots
,i_{m}=1}^{n}\pm z_{i_{1}}^{(1)}\cdots z_{i_{m}}^{(m)}  \label{ii887}
\end{equation}%
such that 
\begin{equation*}
\Vert A\Vert \leq C_{m}n^{\frac{1}{2}+m\alpha (p)}.
\end{equation*}
\end{gkszi}

Henceforth, for all $p\in \lbrack 1,\infty ]$, we represent its conjugate
number by $p^{\ast }$, i.e., $\frac{1}{p}+\frac{1}{p^{\ast }}=1$. We will
also use the following notation for the best known upper estimates of the
Bohnenblust--Hille inequality: 
\begin{equation*}
\begin{array}{ll}
\eta _{\mathbb{C},m} := \prod\limits_{j=2}^{m}\Gamma \left( 2-\frac{1}{j}%
\right) ^{\frac{j}{2-2j}}, & \vspace{0.2cm} \\ 
\eta _{\mathbb{R},m} := \prod\limits_{j=2}^{m}2^{\frac{1}{2j-2}} & \text{for 
}m\leq 13, \vspace{0.2cm} \\ 
\eta _{\mathbb{R},m} := 2^{\frac{446381}{55440}-\frac{m}{2}
}\prod\limits_{j=14}^{m}\left( \frac{\Gamma \left( \frac{3}{2}-\frac{1}{j}
\right) }{\sqrt{\pi }}\right) ^{\frac{j}{2-2j}} & \text{for }m\geq 14.%
\end{array}%
\end{equation*}

Let $1\leq q\leq r\leq \infty $ and $E$ be a Banach space$.$ We recall that
an $m$-linear form $S:E\times \cdots \times E\rightarrow \mathbb{K}$ is
called multiple $\left( r;q\right) $--summing if there is a constant $C>0$
such that 
\begin{equation*}
\textstyle\left\Vert ( S(x_{j_{1}}^{(1)},...,x_{j_{m}}^{(m)}))
_{j_{1},...,j_{m}=1}^{n}\right\Vert _{\ell _{r}}\leq C\sup\limits_{\varphi
\in B_{E^{\ast }}}\left( \sum\limits_{j=1}^{n}\vert \varphi
(x_{j}^{(1)})\vert ^{q}\right) ^{\frac{1}{q}}\cdots \sup\limits_{\varphi \in
B_{E^{\ast }}}\left( \sum\limits_{j=1}^{n}\vert \varphi (x_{j}^{(m)})\vert
^{q}\right) ^{\frac{1}{q}}
\end{equation*}%
for all positive integers $n$.


\subsection{Second part: the proof}

(a) Let us consider first $\left( r,p\right) \in \lbrack 1,2]\times \lbrack
2,2m).$ From now on $T:\ell _{p}^{n}\times \cdots \times \ell
_{p}^{n}\rightarrow \mathbb{K}$ is an $m$-linear form. Since 
\begin{equation*}
\textstyle\sup\limits_{\varphi \in B_{\left( \ell _{p}^{n}\right) ^{\ast
}}}\sum\limits_{j=1}^{n}\left\vert \varphi (e_{j})\right\vert =nn^{-\frac{1}{%
p^{\ast }}}=n^{\frac{1}{p}}
\end{equation*}%
and since $T$ is multiple $\left( 2m/(m+1);1\right) $-summing (recall that
from the Bohnenblust--Hille inequality we know that all continuous $m$%
-linear forms are multiple $( 2m/(m+1);1) $-summing with constant $\eta _{%
\mathbb{K},m}$), we conclude that 
\begin{equation}
\textstyle\left( \sum\limits_{j_{1},...,j_{m}=1}^{n}\left\vert
T(e_{j_{1}},...,e_{j_{m}})\right\vert ^{\frac{2m}{m+1}}\right) ^{\frac{m+1}{%
2m}}\leq \eta _{\mathbb{K},m}\left\Vert T\right\Vert n^{\frac{m}{p}}.
\label{333}
\end{equation}

Therefore, if $1\leq r<2m/(m+1)$, using the H\"{o}lder inequality and (\ref%
{333}), we have 
\begin{equation*}
\begin{array}{l}
\textstyle\left( \sum\limits_{j_{1},...,j_{m}=1}^{n}\left\vert
T(e_{j_{1}},...,e_{j_{m}})\right\vert ^{r}\right) ^{\frac{1}{r}}\vspace{0.2cm%
} \\ 
\textstyle\leq \left( \sum\limits_{j_{1},...,j_{m}=1}^{n}\left\vert
T(e_{j_{1}},...,e_{j_{m}})\right\vert ^{\frac{2m}{m+1}}\right) ^{\frac{m+1}{%
2m}}\left( \sum\limits_{j_{1},...,j_{m}=1}^{n}\left\vert 1\right\vert ^{%
\frac{2mr}{2m-rm-r}}\right) ^{\frac{2m-rm-r}{2mr}}\vspace{0.2cm} \\ 
\textstyle=\left( \sum\limits_{j_{1},...,j_{m}=1}^{n}\left\vert
T(e_{j_{1}},...,e_{j_{m}})\right\vert ^{\frac{2m}{m+1}}\right) ^{\frac{m+1}{%
2m}}\left( n^{m}\right) ^{\frac{2m-rm-r}{2mr}}\vspace{0.2cm} \\ 
\textstyle\leq \eta _{\mathbb{K},m}\left\Vert T\right\Vert n^{\frac{m}{p}}n^{%
\frac{2m-rm-r}{2r}}\vspace{0.2cm} \\ 
\textstyle=\eta _{\mathbb{K},m}n^{\frac{2mr+2mp-mpr-pr}{2pr}}\left\Vert
T\right\Vert .%
\end{array}%
\end{equation*}

Now we consider the case $2m/(m+1)\leq r\leq 2$. From the proof of \cite[%
Theorem 3.2(i)]{apspacea} and from \cite[Theorem 1.1]{ap} we know that, for
all $2m/(m+1)\leq r\leq 2$ and all Banach spaces $E$, every continuous $m$%
-linear operator $S:E\times \cdots \times E\rightarrow \mathbb{K}$ is
multiple $\left( r;2mr/(mr+2m-r)\right) $-summing with constant $(\sigma _{%
\mathbb{K}})^{\frac{(m-1)(mr+r-2m)}{r}}(\eta _{\mathbb{K},m})^{\frac{2m-rm}{r%
}}$ where $\sigma _{\mathbb{R}}=\sqrt{2}$ and $\sigma _{\mathbb{K}}=2/\sqrt{%
\pi }$. Therefore 
\begin{equation}
\begin{array}{l}
\textstyle\left( \sum\limits_{j_{1},...,j_{m}=1}^{n}\left\vert
T(e_{j_{1}},...,e_{j_{m}})\right\vert ^{r}\right) ^{\frac{1}{r}}\vspace{0.2cm%
} \\ 
\textstyle\leq (\sigma _{\mathbb{K}})^{\frac{(m-1)(mr+r-2m)}{r}}(\eta _{%
\mathbb{K},m})^{\frac{2m-rm}{r}}\left\Vert T\right\Vert \left[ \left(
\sup\limits_{\varphi \in B_{\left( \ell _{p}^{n}\right) ^{\ast
}}}\sum\limits_{j=1}^{n}\left\vert \varphi (e_{j})\right\vert ^{\frac{2mr}{%
mr+2m-r}}\right) ^{\frac{mr+2m-r}{2mr}}\right] ^{m}.%
\end{array}
\label{67}
\end{equation}%
Since $1\leq 2mr/(mr+2m-r)\leq 2m/(2m-1)=\left( 2m\right) ^{\ast }<p^{\ast },
$ we have 
\begin{eqnarray}
\textstyle\left( \sup\limits_{\varphi \in B_{\left( \ell _{p}^{n}\right)
^{\ast }}}\sum\limits_{j=1}^{n}\left\vert \varphi \left( e_{j}\right)
\right\vert ^{\frac{2mr}{mr-r+2m}}\right) ^{\frac{mr-r+2m}{2mr}} &=&%
\textstyle(n(n^{-\frac{1}{p^{\ast }}})^{\frac{2mr}{mr-r+2m}})^{\frac{mr-r+2m%
}{2mr}}  \notag \\
&=&\textstyle n^{\frac{2mr+2mp-mpr-pr}{2mpr}}  \label{76}
\end{eqnarray}%
%
%
%
%
%
%
%
%
%
%
%
%
and finally, from (\ref{67}) and (\ref{76}), we obtain 
\begin{equation*}
\textstyle\left( \sum\limits_{j_{1},...,j_{m}=1}^{n}\left\vert
T(e_{j_{1}},...,e_{j_{m}})\right\vert ^{r}\right) ^{\frac{1}{r}}\leq (\sigma
_{\mathbb{K}})^{\frac{(m-1)(mr+r-2m)}{r}}(\eta _{\mathbb{K},m})^{\frac{2m-rm%
}{r}}n^{\frac{2mr+2mp-mpr-pr}{2pr}}\left\Vert T\right\Vert .
\end{equation*}

Now we prove the optimality of the exponents. Suppose that the theorem is
valid for an exponent $s$, i.e., 
\begin{equation*}
\textstyle\left( \sum\limits_{j_{1},...,j_{m}=1}^{n}\left\vert
T(e_{j_{1}},...,e_{j_{m}})\right\vert ^{r}\right) ^{\frac{1}{r}}\leq
D_{m,r,p}^{\mathbb{K}}n^{s\allowbreak }\left\Vert T\right\Vert .
\end{equation*}%
Since $p\geq 2,$ from the Generalized Kahane--Salem--Zygmund inequality
(using the $m$-linear form (\ref{ii887})) we have 
\begin{equation*}
\textstyle n^{\frac{m}{r}}\leq C_{m}D_{m,r,p}^{\mathbb{K}}n^{s}n^{\frac{m+1}{%
2}-\frac{m}{p}}
\end{equation*}%
and thus, making $n\rightarrow \infty $, we obtain%
\begin{equation*}
\textstyle s\geq \frac{2mr+2mp-mpr-pr}{2pr}.
\end{equation*}

The case $\left( r,p\right) \in \lbrack 1,2mp/(mp+p-2m)]\times \lbrack
2m,\infty ]$ is analogous. In fact, from the
Hardy--Littlewood/Praciano-Pereira inequality and \cite[Theorem 1.1]{ap} we
know that%
\begin{equation}
\textstyle\left( \sum\limits_{j_{1},...,j_{m}=1}^{n}\left\vert
T(e_{j_{1}},...,e_{j_{m}})\right\vert ^{\frac{2mp}{mp+p-2m}}\right) ^{\frac{%
mp+p-2m}{2mp}}\leq (\sigma _{\mathbb{K}})^{\frac{2m\left( m-1\right) }{p}%
}(\eta _{\mathbb{K},m})^{\frac{p-2m}{p}}\left\Vert T\right\Vert .
\label{3333}
\end{equation}%
Therefore, from H\"{o}lder's inequality and (\ref{3333}), we have 
\begin{equation}
\begin{array}{l}
\textstyle\left( \sum\limits_{j_{1},...,j_{m}=1}^{n}\left\vert
T(e_{j_{1}},...,e_{j_{m}})\right\vert ^{r}\right) ^{\frac{1}{r}}\vspace{0.2cm%
} \\ 
\textstyle\leq \left( \sum\limits_{j_{1},...,j_{m}=1}^{n}\left\vert
T(e_{j_{1}},...,e_{j_{m}})\right\vert ^{\frac{2mp}{mp+p-2m}}\right) ^{\frac{%
mp+p-2m}{2mp}}\left( \sum\limits_{j_{1},...,j_{m}=1}^{n}\left\vert
1\right\vert ^{\frac{2mpr}{2mp+2mr-mpr-pr}}\right) ^{\frac{2mp+2mr-mpr-pr}{%
2mpr}}\vspace{0.2cm} \\ 
\textstyle=\left( \sum\limits_{j_{1},...,j_{m}=1}^{n}\left\vert
T(e_{j_{1}},...,e_{j_{m}})\right\vert ^{\frac{2mp}{mp+p-2m}}\right) ^{\frac{%
mp+p-2m}{2mp}}\left( n^{m}\right) ^{\frac{2mp+2mr-mpr-pr}{2mpr}}\vspace{0.2cm%
} \\ 
\textstyle\leq (\sigma _{\mathbb{K}})^{\frac{2m\left( m-1\right) }{p}}(\eta
_{\mathbb{K},m})^{\frac{p-2m}{p}}n^{\frac{2mp+2mr-mpr-pr}{2pr}}\vspace{0.2cm}%
\left\Vert T\right\Vert .%
\end{array}
\label{um8}
\end{equation}%
Since $p\geq 2m$, the optimality of the exponent is obtained \textit{ipsis
litteris} as in the previous case.

If $\left( r,p\right) \in (2mp/(mp+p-2m),\infty )\times \lbrack 2m,\infty ]$
we have 
\begin{equation*}
\textstyle\frac{2mr+2mp-mpr-pr}{2pr}<0
\end{equation*}%
and 
\begin{eqnarray*}
\textstyle\left( \sum\limits_{j_{1},...,j_{m}=1}^{n}\left\vert
T(e_{j_{1}},...,e_{j_{m}})\right\vert ^{r}\right) ^{\frac{1}{r}} &\leq &%
\textstyle\left( \sum\limits_{j_{1},...,j_{m}=1}^{n}\left\vert
T(e_{j_{1}},...,e_{j_{m}})\right\vert ^{\frac{2mp}{mp+p-2m}}\right) ^{\frac{%
mp+p-2m}{2mp}} \\
&\leq &\textstyle D_{m,\frac{2mp}{mp+p-2m},p}^{\mathbb{K}}\left\Vert
T\right\Vert \\
&=&\textstyle D_{m,\frac{2mp}{mp+p-2m},p}^{\mathbb{K}}\left\Vert
T\right\Vert n^{\max \left\{ \frac{2mr+2mp-mpr-pr}{2pr},0\right\} }.
\end{eqnarray*}%
In this case the optimality of the exponent $\max \left\{
(2mr+2mp-mpr-pr)/2pr,0\right\} $ is immediate, since one can easily verify
that no negative exponent of $n$ is possible.

(b) Let us first consider $\left( r,p\right) \in \lbrack 2,p/(p-m)]\times
(m,2m].$ Define 
\begin{equation*}
\textstyle q=\frac{mr}{r-1}
\end{equation*}%
and note that $q\leq 2m$ and $r=q/(q-m)$. Since $q/(q-m)=r\leq p/(p-m)$ we
have $p\leq q.$ Then $m<p\leq q\leq 2m.$ Note that 
\begin{equation*}
\textstyle q^{\ast }=\frac{mr}{mr+1-r}.
\end{equation*}%
Since $m<q\leq 2m$, by the Hardy-Littlewood/Dimant-Sevilla-Peris inequality
and using \cite[Section 5]{dimant} we know that every continuous $m$-linear
operator on any Banach space $E$ is multiple $\left( q/(q-m);q^{\ast
}\right) $--summing with constant $\textstyle\left( \sqrt{2}\right) ^{m-1}$,
i.e., multiple $\left( r;mr/(mr+1-r)\right) $--summing with constant $%
\textstyle\left( \sqrt{2}\right) ^{m-1}$. So for $T:\ell _{p}^{n}\times
\cdots \times \ell _{p}^{n}\rightarrow \mathbb{K}$ we have (since $q^{\ast
}\leq p^{\ast })$, 
\begin{equation*}
\begin{array}{l}
\textstyle\left( \sum\limits_{j_{1},...,j_{m}=1}^{n}\left\vert T\left(
e_{j_{1}},...,e_{j_{m}}\right) \right\vert ^{r}\right) ^{\frac{1}{r}}\vspace{%
0.2cm} \\ 
\leq \textstyle\left( \sqrt{2}\right) ^{m-1}\left\Vert T\right\Vert \left[
\left( \sup\limits_{\varphi \in B_{\left( \ell _{p}^{n}\right) ^{\ast
}}}\sum\limits_{j=1}^{n}\left\vert \varphi \left( e_{j}\right) \right\vert ^{%
\frac{mr}{mr+1-r}}\right) ^{\frac{mr+1-r}{mr}}\right] ^{m}\vspace{0.2cm} \\ 
=\textstyle\left( \sqrt{2}\right) ^{m-1}\left\Vert T\right\Vert \left[
(n(n^{-\frac{1}{p^{\ast }}})^{\frac{mr}{mr+1-r}})^{\frac{mr+1-r}{mr}}\right]
^{m}\vspace{0.2cm} \\ 
=\textstyle\left( \sqrt{2}\right) ^{m-1}\left\Vert T\right\Vert n^{\frac{%
p+mr-rp}{pr}}.%
\end{array}%
\end{equation*}

Note that if we have tried to use above an argument similar to (\ref{um8}),
via H\"{o}lder's inequality, we would obtain worse exponents. Now we prove
the optimality following the lines of \cite{dimant}. Defining $R:\ell
_{p}^{n}\times \cdots \times \ell _{p}^{n}\rightarrow \mathbb{K}$ by $%
R(x^{(1)},...,x^{(m)})=\sum_{j=1}^{n}x_{j}^{(1)}\cdots x_{j}^{(1)}$, from H%
\"{o}lder's inequality we can easily verify that%
\begin{equation*}
\textstyle\left\Vert R\right\Vert \leq n^{1-\frac{m}{p}}.
\end{equation*}%
So if the theorem holds for $n^{s},$ plugging the $m$-linear form $R$ into
the inequality we have 
\begin{equation*}
\textstyle n^{\frac{1}{r}}\leq D_{m,r,p}^{\mathbb{K}}n^{s}n^{1-\frac{m}{p}}
\end{equation*}%
and thus, by making $n\rightarrow \infty $, we obtain 
\begin{equation*}
\textstyle s\geq \frac{p+mr-rp}{pr}.
\end{equation*}

If $\left( r,p\right) \in (p/(p-m),\infty )\times (m,2m]$ we have 
\begin{equation*}
\textstyle\frac{p+mr-rp}{pr}<0
\end{equation*}%
and 
\begin{eqnarray*}
\textstyle\left( \sum\limits_{j_{1},...,j_{m}=1}^{n}\left\vert
T(e_{j_{1}},...,e_{j_{m}})\right\vert ^{r}\right) ^{\frac{1}{r}} &\leq &%
\textstyle\left( \sum\limits_{j_{1},...,j_{m}=1}^{n}\left\vert
T(e_{j_{1}},...,e_{j_{m}})\right\vert ^{\frac{p}{p-m}}\right) ^{\frac{p-m}{p}%
} \\
&\leq &\textstyle\left( \sqrt{2}\right) ^{m-1}\left\Vert T\right\Vert \\
&=&\textstyle\left( \sqrt{2}\right) ^{m-1}\left\Vert T\right\Vert n^{\max
\left\{ \frac{p+mr-rp}{pr},0\right\} }
\end{eqnarray*}%
In this case the optimality of the exponent $\max \left\{
(p+mr-rp)/pr,0\right\} $ is immediate, since one can easily verify that no
negative exponent of $n$ is possible.

\begin{remark}
Observing the proof of Theorem \ref{888} we conclude that the optimal
constant $D_{m,r,p}^{\mathbb{K}}$ satisfies:

\begin{equation*}
D_{m,r,p}^{\mathbb{K}}\leq \left\{ 
\begin{array}{ll}
\textstyle\eta _{\mathbb{K},m} & \text{if }\textstyle\left( r,p\right) \in %
\left[ 1,\frac{2m}{m+1}\right] \times \lbrack 2,2m),\vspace{0.2cm} \\ 
(\sigma _{\mathbb{K}})^{\frac{(m-1)(mr+r-2m)}{r}}(\eta _{\mathbb{K},m})^{%
\frac{2m-rm}{r}} & \text{if }\left( r,p\right) \in \left( \frac{2m}{m+1},2%
\right] \times \lbrack 2,2m),\vspace{0.2cm} \\ 
(\sigma _{\mathbb{K}})^{\frac{2m\left( m-1\right) }{p}}(\eta _{\mathbb{K}%
,m})^{\frac{p-2m}{p}} & \text{if }\left( r,p\right) \in \lbrack 1,\infty
)\times \lbrack 2m,\infty ],\vspace{0.2cm} \\ 
(\sqrt{2})^{m-1} & \text{if }\left( r,p\right) \in \lbrack 2,\infty )\times
(m,2m].%
\end{array}%
\right. 
\end{equation*}%
If $p>2m^{3}-4m^{2}+2m,$ using \cite{ap2} we can improve $(\sigma _{\mathbb{%
K}})^{\frac{2m\left( m-1\right) }{p}}(\eta _{\mathbb{K},m})^{\frac{p-2m}{p}}$
to $\eta _{\mathbb{K},m}$.
\end{remark}

\section{Final comments and results}

In this section we obtain partial answers for the cases not covered by our
main theorem, i.e., the cases $\left( r,p\right) \in \lbrack 1,2]\times
\lbrack 1,2)$ and $\left( r,p\right) \in ( 2,\infty )\times \lbrack 1,m]$.

\begin{proposition}
\label{prop1} Let $m\geq 2$ be a positive integer.

\begin{itemize}
\item[(a)] If $\left( r,p\right) \in \lbrack 1,2]\times \lbrack 1,2)$, then
there is a constant $D_{m,r,p}^{\mathbb{K}}>0$ such that 
\begin{equation}
\textstyle\left( \sum\limits_{j_{1},...,j_{m}=1}^{n}\left\vert
T(e_{j_{1}},...,e_{j_{m}})\right\vert ^{r}\right) ^{\frac{1}{r}}\leq
D_{m,r,p}^{\mathbb{K}}n^{\frac{2mr+2mp-mpr-pr}{2pr}}\left\Vert T\right\Vert
\label{iu}
\end{equation}%
for all $m$--linear forms $T:\ell _{p}^{n}\times \cdots \times \ell
_{p}^{n}\rightarrow \mathbb{K}$ and all positive integers $n$. Moreover the
optimal exponent of $n$ is not smaller than $(2m-r)/2r$.

\item[(b)] If $\left( r,p\right) \in (2,\infty )\times \lbrack 1,m]$, then
there is a constant $D_{m,r,p}^{\mathbb{K}}>0$ such that 
\begin{equation*}
\textstyle\left( \sum\limits_{j_{1,.},...,j_{m}=1}^{n}\left\vert
T(e_{j_{1}},...,e_{j_{m}})\right\vert ^{r}\right) ^{\frac{1}{r}}\leq \left\{ 
\begin{array}{ll}
D_{m,r,p}^{\mathbb{K}}n^{\frac{2m-p+\varepsilon }{pr}}\left\Vert T\right\Vert
& \text{if }p>2\vspace{0.2cm} \\ 
D_{m,r,p}^{\mathbb{K}}n^{\frac{2m-p}{pr}}\left\Vert T\right\Vert & \text{if }%
p=2%
\end{array}%
\right.
\end{equation*}%
for all $m$--linear forms $T:\ell _{p}^{n}\times \cdots \times \ell
_{p}^{n}\rightarrow \mathbb{K}$ and all positive integers $n$ and all $%
\varepsilon >0.$ Moreover the optimal exponent of $n$ is not smaller than $%
(2mr+2mp-mpr-pr)/2pr$ if $2\leq p\leq m$. In the case $1\leq p\leq 2$ the
optimal exponent of $n$ is not smaller than $(2m-r)/2r$.
\end{itemize}
\end{proposition}

\begin{proof}
(a) The proof of (\ref{iu}) is the same of the proof of Theorem \ref{888}%
(a). The estimate for the bound of the optimal exponent also uses the
Generalized Kahane--Salem--Zygmund inequality. Since $p\leq 2$ we have 
\begin{equation*}
\textstyle n^{\frac{m}{r}}\leq C_{m}D_{m,r,p}^{\mathbb{K}}n^{s}n^{\frac{1}{2}%
}
\end{equation*}%
and thus, by making $n\rightarrow \infty $, 
\begin{equation*}
\textstyle s\geq \frac{2m-r}{2r}.
\end{equation*}

(b) Let $\delta =0$ if $p=2$ and $\delta >0$ if $p>2$. First note that every
continuous $m$-linear form on $\ell _{p}$ spaces is obviously multiple $%
\left( \infty ;p^{\ast }-\delta \right) $--summing and also multiple $\left(
2;2m/(2m-1)\right) $--summing (this is a consequence of the
Hardy--Littlewood inequality and \cite[Section 5]{dimant}). Using \cite[%
Proposition 4.3]{compl} we conclude that every continuous $m$-linear form on 
$\ell _{p}$ spaces is multiple $\left( r;mpr/(2m+mpr-mr-p+\varepsilon
)\right) $--summing for all $\varepsilon >0$ (and $\varepsilon =0$ if $p=2$%
). Therefore, there exist $D_{m,r,p}^{\mathbb{K}}>0$ such that 
\begin{equation*}
\begin{array}{l}
\textstyle\left( \sum\limits_{j_{1},...,j_{m}=1}^{n}\left\vert
T(e_{j_{1}},...,e_{j_{m}})\right\vert ^{r}\right) ^{\frac{1}{r}}\vspace{0.2cm%
} \\ 
\textstyle\leq D_{m,r,p}^{\mathbb{K}}\left[ (n(n^{-\frac{1}{p^{\ast }}})^{%
\frac{mpr}{2m+mpr-mr-p+\varepsilon }})^{\frac{2m+mpr-mr-p+\varepsilon }{mpr}}%
\right] ^{m}\left\Vert T\right\Vert \vspace{0.2cm} \\ 
\textstyle=D_{m,r,p}^{\mathbb{K}}n^{\frac{2m+mpr-mr-p+\varepsilon }{pr}}(n^{%
\frac{1}{p}-1})^{m}\left\Vert T\right\Vert \vspace{0.2cm} \\ 
\textstyle=D_{m,r,p}^{\mathbb{K}}n^{\frac{2m-p+\varepsilon }{pr}}\left\Vert
T\right\Vert .%
\end{array}%
\end{equation*}%
The bounds for the optimal exponents are obtained via the Generalized
Kahane--Salem--Zygmund inequality as in the previous cases.
\end{proof}

\begin{remark}
\bigskip Item (b) of the Proposition \ref{prop1} with $p=2$ recovers \cite[%
Corollary 5.20(ii)]{compl}.
\end{remark}

We believe that the remaining cases (those in which we do not have achieved
the optimality of the exponents) are interesting for further investigation
trying to have a full panorama, covering all cases with optimal estimates.


\begin{thebibliography}{99}
\bibitem{n} N. Albuquerque, F. Bayart, D. Pellegrino, J.B. Seoane--Sep\'{u}%
lveda, Optimal Hardy--Littlewood type inequalities for polynomials and
multilinear operators, to appear in Israel J. Math. (2015), arXiv:1311.3177
[math.FA] 13 Nov 2013.

\bibitem{alb} N. Albuquerque, F. Bayart, D. Pellegrino, J.B. Seoane--Sep\'{u}%
lveda, Sharp generalizations of the multilinear Bohnenblust--Hille
inequality, J. Funct. Anal. \textbf{266} (2014), 3726--3740.

\bibitem{ap2} G. Ara\'{u}jo, D. Pellegrino, On the constants of the
Bohnenblust--Hille and Hardy--Littlewood inequalities, arXiv:1407.7120
[math.FA], 26 Jul 2014.

\bibitem{apspacea} G. Ara\'{u}jo, D. Pellegrino, Spaceability and optimal
estimates for summing multilinear operators, arXiv:1403.6064v2 [math.FA] 26
Sep 2014.

\bibitem{ap} G. Ara\'{u}jo, D. Pellegrino, D.D.P. Silva, On the upper bounds
for the constants of the Hardy--Littlewood inequality, J. Funct. Anal. 
\textbf{267} (2014), 1878--1888.

\bibitem{bohr} F. Bayart, D. Pellegrino, J.B. Seoane-Sep\'{u}lveda, The Bohr
radius of the $n$--dimensional polydisc is equivalent to $\sqrt{(\log n)/n}$%
, Adv. Math. \textbf{264} (2014) 726--746.


\bibitem{boas} H.P. Boas, Majorant series, J. Korean Math. Soc. \textbf{37}
(2000), 321--337.

\bibitem{BK97} H.P. Boas, The football player and the infinite series,
Notices Amer. Math. Soc. \textbf{44} (1997), no. 11, 1430--1435.

\bibitem{bh} H.F. Bohnenblust, E. Hille, On the absolute convergence of
Dirichlet series, Ann. of Math. \textbf{32} (1931), 600--622.











\bibitem{compl} G. Botelho, C. Michels, D. Pellegrino, Complex interpolation
and summability properties of multilinear operators, Revista Matem\'{a}tica
Complutense \textbf{23} (2010), 139--161.

\bibitem{camposar}J.R. Campos, G.A. Munoz-Fernández, D. Pellegrino and J.B. Seoane-Sepúlveda, On the optimality of teh complex Bohnenblust--HIlle inequality, arXiv 1301.1539v3 [mathFA] 11 Sep 2013.

\bibitem{cobos} F. Cobos, T. K\"{u}hn, J. Peetre, On $\mathfrak{G}_p$%
-classes of trilinear forms, J. London Math. Soc. (2) \textbf{59} (1999),
1003--1022


\bibitem{deff2} A. Defant, D. Popa, U. Schwarting, Coordinatewise multiple
summing operators in Banach spaces, J. Funct. Anal. \textbf{259} (2010), no.
1, 220--242.

\bibitem{deff} A. Defant, P. Sevilla-Peris, A new multilinear insight on
Littlewood's 4/3-inequality. J. Funct. Anal. \textbf{256} (2009), no. 5,
1642--1664.

\bibitem{sur} A. Defant, P. Sevilla-Peris, The Bohnenblust-Hille cycle of
ideas from a modern point of view. Funct. Approx. Comment. Math. \textbf{50}
(2014), no. 1, 55--127.

\bibitem{Di} J. Diestel, H. Jarchow, A. Tonge, Absolutely summing operators,
Cambridge University Press, Cambridge, 1995.

\bibitem{dimant} V. Dimant, P. Sevilla--Peris, Summation of coefficients of
polynomials on $\ell _{p}$ spaces, arXiv:1309.6063v1 [math.FA] 24 Sep 2013.

\bibitem{diniz} D. Diniz, G.A. Mu\~{n}oz-Fern\'{a}dez, D. Pellegrino, and
J.B. Seoane-Sep\'{u}lveda, Lower Bounds for the constants in the
Bohnenblust-Hille inequality: the case of real scalars, Proc. Amer. Math.
Soc. \textbf{142} (2014), n. 2, 575--580.



\bibitem{hardy} G. Hardy, J.E. Littlewood, Bilinear forms bounded in space $%
[p,q]$, Quart. J. Math. \textbf{5} (1934), 241--254.




\bibitem{ddss} D. Nu\~{n}ez-Alarc\'{o}n, D. Pellegrino, J.B. Seoane-Sep\'{u}%
lveda, On the Bohnenblust--Hille inequality and a variant of Littlewood's $%
4/3$ inequality, J. Funct. Anal. \textbf{264} (2013), 326--336.






\bibitem{Nuuu} D. Nu\~{n}ez-Alarc\'{o}n, D. Pellegrino, J.B. Seoane-Sep\'{u}%
lveda, D.M. Serrano-Rodr{\'{\i}}guez, There exist multilinear
Bohnenblust-Hille constants $(C_{n})_{n=1}^{\infty }$ with $%
\lim_{n\rightarrow \infty }(C_{n+1}-C_{n})=0$, J. Funct. Anal. \textbf{264}
(2013), 429--463.












\bibitem{pra} T. Praciano-Pereira, On bounded multilinear forms on a class
of $\ell _{p}$ spaces. J. Math. Anal. Appl. \textbf{81} (1981), 561--568.

\end{thebibliography}
\end{document}